\documentclass[12pt]{article}
\usepackage{amsmath,amssymb,amsthm,graphicx,cite}

\newtheorem{lemma}{Lemma}
\newtheorem{theorem}{Theorem}

\renewcommand{\phi}{\varphi}

\newcommand{\R}{\mathbb{R}}
\newcommand{\Sym}{\mathbb{S}}

\newcommand{\tr}{\mathrm{tr}}

\title{Compact convex sets with prescribed facial dimensions}
\author{Vera Roshchina\thanks{School of Science, RMIT University, vera.roshchina@rmit.edu.au}, Tian Sang\thanks{School of Science, RMIT University, s3556268@student.rmit.edu.au} and  David Yost\thanks{Centre for Informatics and Applied Optimisation, Federation University, d.yost@federation.edu.au}}

\begin{document}

\maketitle

\begin{abstract}
While faces of a polytope form a well structured lattice, in which faces of each possible dimension are present, this is not true for general compact convex sets. We address the question of what dimensional patterns are possible for the faces of general closed convex sets. We show that for any finite sequence of positive integers there exist compact convex sets which only have extreme points and faces with dimensions from this prescribed sequence. We also discuss another approach to dimensionality, considering the dimension of the union of all faces of the same dimension. We show that the questions arising from this approach are highly nontrivial and give examples of convex sets for which the sets of extreme points have fractal dimension.
\end{abstract}

\section{Introduction}

It is well known that faces of polyhedral sets have a well-defined structure (see \cite[Chap.~2]{Ziegler}). In particular, every face of a polyhedral set is a polyhedron, and there are no `gaps' in the dimensions of their faces. On the other hand, a simple reformulation of \cite[Corollary~3.7]{HillWaters} asserts that in the compact convex set of all positive semidefinite $n\times n$ matrices with trace 1, every proper face has dimension $k^2-1$ for some $k<n$. Thus there are naturally occuring examples with serious gaps in the dimensions of their faces. For other descriptions of this phenomenon, see Theorem 2.25 and the explanation that follows it in \cite{Tuncel} (for the cone $\Sym_+^n$ of positive semidefinite $n\times n$ matrices), or \cite[Theorem~5.36]{StateSpaces} (for the state space of a $C^*$-algebra). This raises the question, what are the possible patterns for the dimensions of faces of compact convex sets?


Recall that a {\em face} $F$ of a closed convex set $C\subset \R^n$ is a closed convex subset of $C$ such that for any point $x\in F$ and for any line segment $[a,b]\subset C$ such that $x\in (a,b)$, we have $a,b\in F$. The fact that $F$ is a face of $C$ is expressed as $F\lhd C$.

The difference between this definition and the definition of faces of polyhedral sets as intersections with supporting hyperplanes is due to the fact that for nonpolyhedral convex sets faces are not necessarily {\em exposed}: it may happen that a face cannot be represented as the intersection of a supporting hyperplane with the set. Some classic examples are shown in Figs.~\ref{fig:Tablet} (see \cite{Rockafellar}) and \ref{fig:nonexp2d} (see \cite{Pataki}).
\begin{figure}[ht]
	\centering
	\includegraphics[height = 80pt]{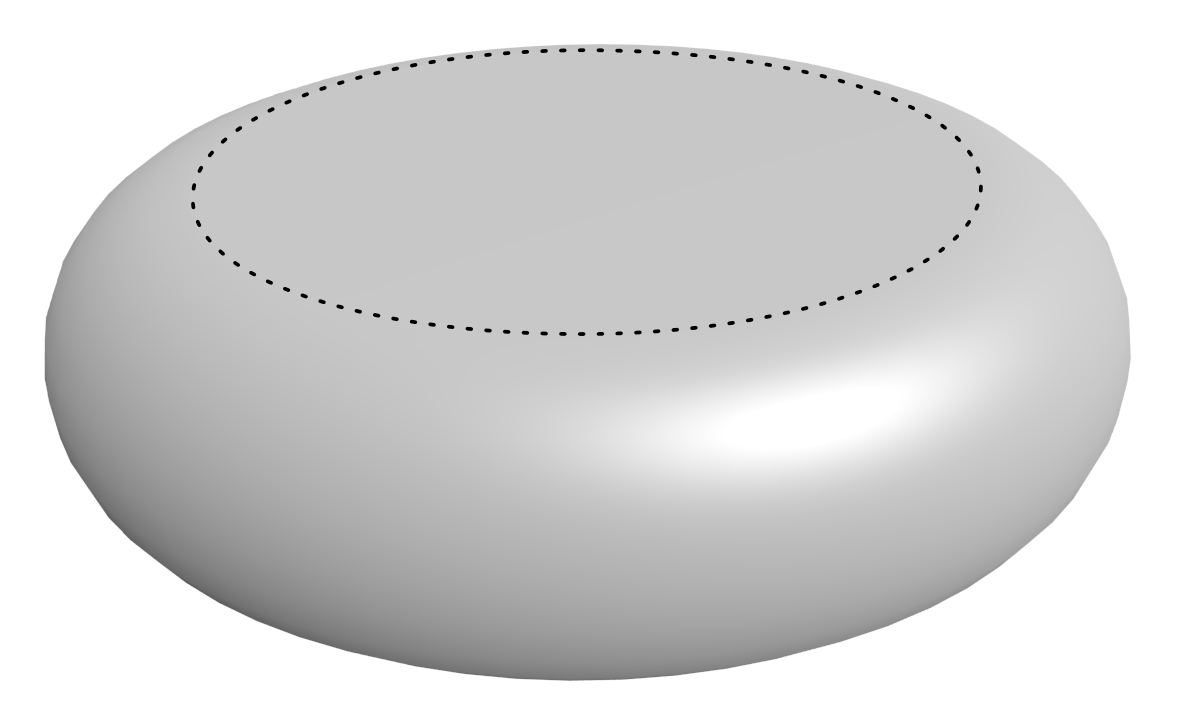}
	\caption{Convex hull of a torus is not facially exposed (the dashed line shows the unexposed extreme points).}
	\label{fig:Tablet}
\end{figure}
\begin{figure}[ht]
	\centering
	\includegraphics[height = 100pt]{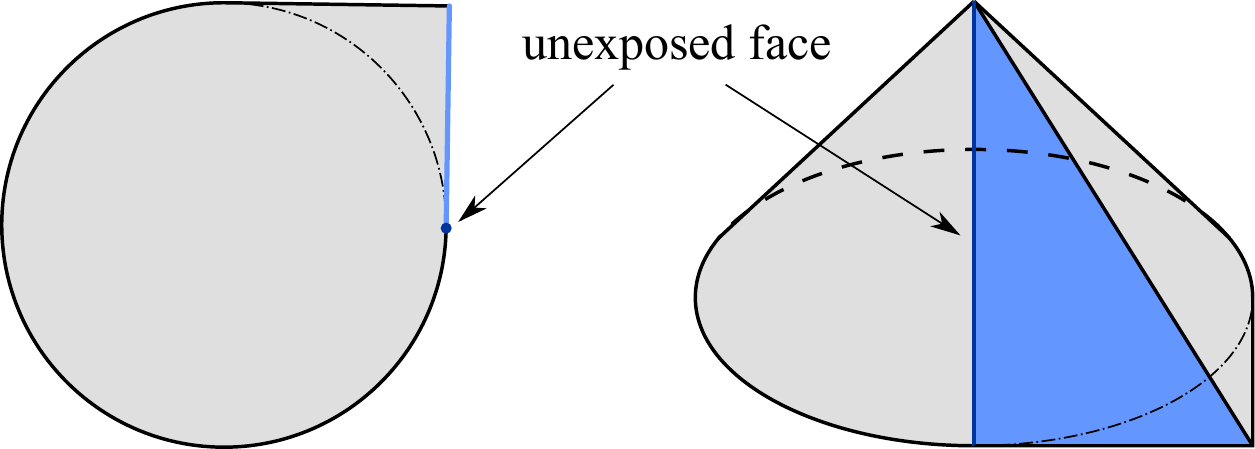}
	\caption{An example of a two dimensional set and a three dimensional cone that have an unexposed face.}
	\label{fig:nonexp2d}
\end{figure}

The dimension of a convex set is the dimension of its affine hull, same for the face. We refer the reader to the classic textbooks \cite{Rockafellar,JBHU}. We also would like to mention that some problems related to dimensions of convex sets were studied in the literature. For instance, \cite{Eckhardt} focusses on the dimensions of convex sets coming from optimisation problems with inequality constraints, and \cite{GrunbaumIntersection} deals with the results related to the dimensions of intersections of convex sets. However, we were unable to identify references that would address the existence of convex sets with prescribed facial dimensions.

The total number of possible face patterns in $n$ dimensional space is the cardinality of the powerset of $n$ elements. This is because every set contains zero-dimensional faces (because of the Krein-Milman theorem). We can write down face patterns either as  an increasing  sequence of positive numbers $(d_1, d_2, \dots, d_k)$, which encode all possible dimensions of faces of positive dimension present in a set, or as a binary sequence $(b_1, b_2, \dots, b_n)$, where $b_i=1$ if a face of dimension $i$ is present in the set, and $b_i=0$ otherwise. For example, the dimensional pattern of a tetrahedron is either $(1,2,3)$ in the $d$-notation or $(1,1,1)$ in the binary notation, and the pattern of a closed Euclidean ball is either $(n)$ or $(0,0,\dots, 1)$, as it does not have any faces except for zero- and $n$-dimensional ones. We will use the first encoding style via an increasing sequence of positive numbers in what follows.

The easiest cases to classify are the ones that we can visualise, i.e. the convex compact sets in zero- one-, two- and three-dimensional  spaces. In dimension zero we have singletons $\{x\}$ for any real $x$ with pattern $()$, in one-dimensional space there is no freedom: the only fully dimensional convex compact sets are line segments, with the only possible pattern $(1)$. On the plane the two-dimensional possibilities are exhausted by a circle and a triangle, with patterns $(2)$ and $(1,2)$ respectively (see Fig.~\ref{fig:2DPatterns}). Therefore for the two dimensional case we have four possibilities: $()$, $(1)$, $(1,2)$ and $(2)$, which coincides with the cardinality of the powerset of two: $2^2 = 4$.
\begin{figure}[ht]
	\centering\includegraphics[height=90pt]{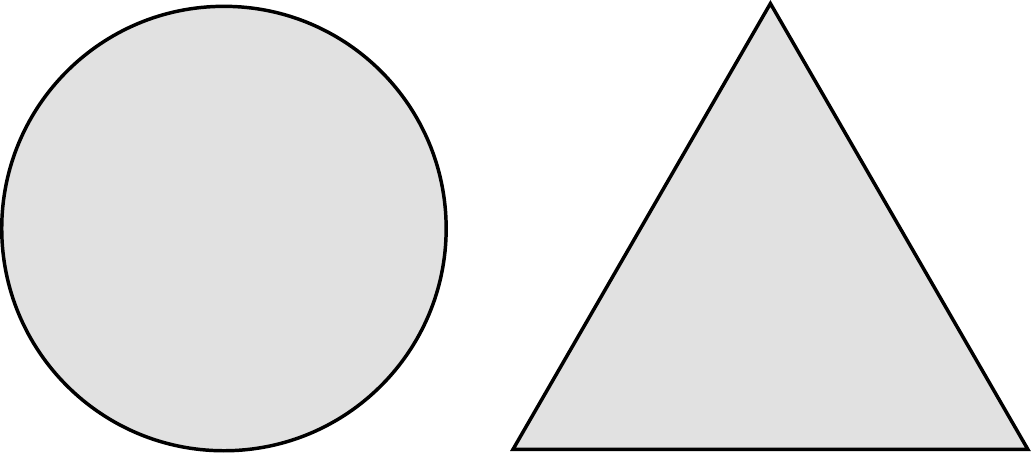}
	\caption{All possible face patterns of fully dimensional sets in two dimensional case are given by a disk and a triangle.}
	\label{fig:2DPatterns}
\end{figure}

In three dimensions the possibilities for fully dimensional sets are exhausted by the unit ball (3), the tetrahedron (1,2,3), the unit ball intersected with a closed half-space (2,3), and the convex hull of a circle in the plane and two points on opposite sides of the plane (1,3) (see Fig.~\ref{fig:ThreeDimensionsS}), together with the lower dimensional examples we have in total $2^3= 8$ possibilities.
\begin{figure}[ht]
	\centering\includegraphics[height=90pt]{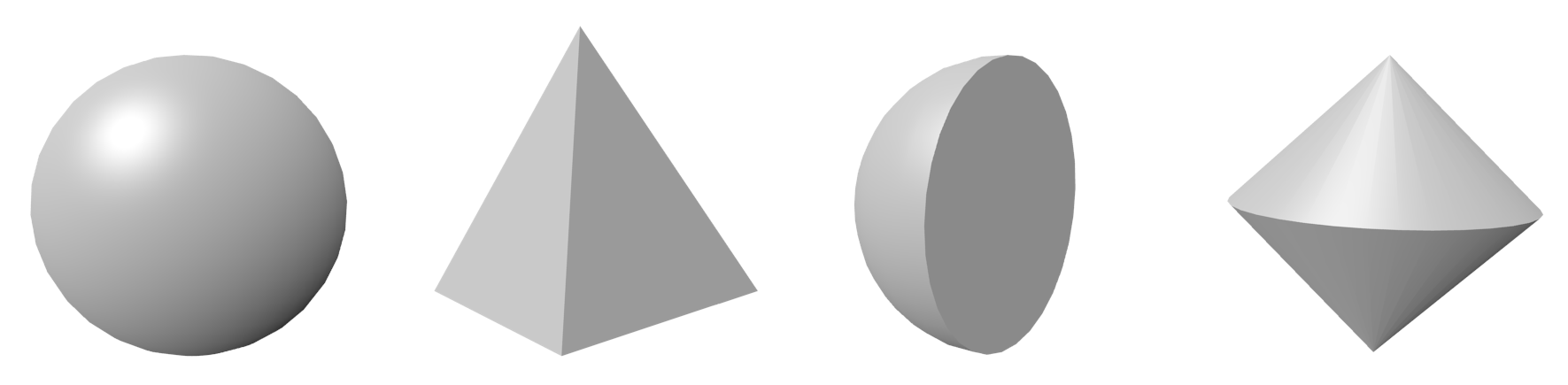}
	\caption{All possible facial patterns for the three dimensional sets}
	\label{fig:ThreeDimensionsS}
\end{figure}

\section{Main Result}

We show that all patterns of facial dimensions can be realised in a compact convex set.

\begin{theorem}\label{thm:everythingispossible}
For any increasing sequence of positive integers
$$
d = (d_1,d_2, \dots, d_k)
$$
there exists a compact convex set in $d_k$-dimensional space such that the vector $d$ describes the pattern of facial dimensions for this set.
\end{theorem}

To prove this, we need the following technical lemma, which is surely known, but we were not able to identify it in the literature. We hence provide a short proof here as well.

\begin{lemma}\label{lem:sumoffaces} Let $P,Q\subset \R^n$ be nonempty convex compact sets, and let $C=P+Q$. Then every face of  $C$ is the Minkowski sum of faces of $P$ and $Q$. More precisely,
$$
\forall \, F\lhd C \quad \exists F_P\lhd P, F_Q\lhd Q \text{ such that } F = F_P+F_Q.
$$
\end{lemma}
\begin{proof} Let $F$ be a nonempty face of $C$. We construct two sets
$$
F_P:= \{x\in P\,|\, \exists y \in Q, x+y \in F\}, \qquad F_Q:= \{y\in Q\,|\, \exists x \in P, x+y \in F\}.
$$
Both $F_P$ and $F_Q$ are nonempty since $F$ is nonempty.

First we show that $F=F_P+F_Q$. It is obvious that $F\subset F_P+F_Q$, and it remains to show the reverse inclusion. For that, pick an arbitrary $x\in F_P$, $y\in F_Q$. We will next show that $z=x+y\in F$.

By the definition of $F_P$ and $F_Q$ there exist $u\in P$ and $v\in Q$ such that $x+v\in F$ and $y+u\in F$. If $x=u$ or $y=v$, there is nothing to prove, as in this case $z=u+v\in F$. Otherwise, by the convexity of $F$ we have
$$
z' = \frac{x+v}{2}+\frac{y+u}{2}\in F.
$$
At the same time, notice that $x+y\in P+Q\subset C$; likewise, $u+v\in P+Q \subset C$, and $z'\in (x+y, u+v)$. Since $F$ is a face of $C$, this yields $z=x+y\in F$.

It remains to show that both $F_P$ and $F_Q$ are faces of $P$ and $Q$ respectively. First note that both are closed compact sets, and that $F_Q\subset Q $ and $F_P\subset P$.

Let $x\in F_P$, and pick any interval $[a,b]\subset P$ such that $x\in (a,b)$. By the definition of $C$, for an arbitrary $y\in F_Q$  we have $a+y, b+y \in C$. At the same time, $x+y\in F_P+F_Q = F$ and $x+y\in (a+y,b+y)$. From $F\lhd C$ we have $ [a+y,b+y]\subset F$, hence, $a+y,b+y\in F$, and therefore $a,b\in F_P$. This shows that $F_P$ is a face of $P$. The proof for $F_Q$ is identical.
\end{proof}

In the proof of Theorem~\ref{thm:everythingispossible} presented next we use an inductive argument to explicitly construct a compact convex set with a given facial pattern from a lower dimensional example for a truncated sequence. The key observation is that the Minkowski sum of an arbitrary compact convex set with a unit ball does not generate faces of any new dimensions (compared to the original set) other than possibly the fully dimensional face that coincides with the sum, which follows directly from Lemma~\ref{lem:sumoffaces}. We sketched the Minkowski sum of two simple compact convex sets with a Euclidean ball in  Fig.~\ref{fig:minkowski} to illustrate this argument.
\begin{figure}[ht]
	\centering
	\includegraphics[scale=0.25]{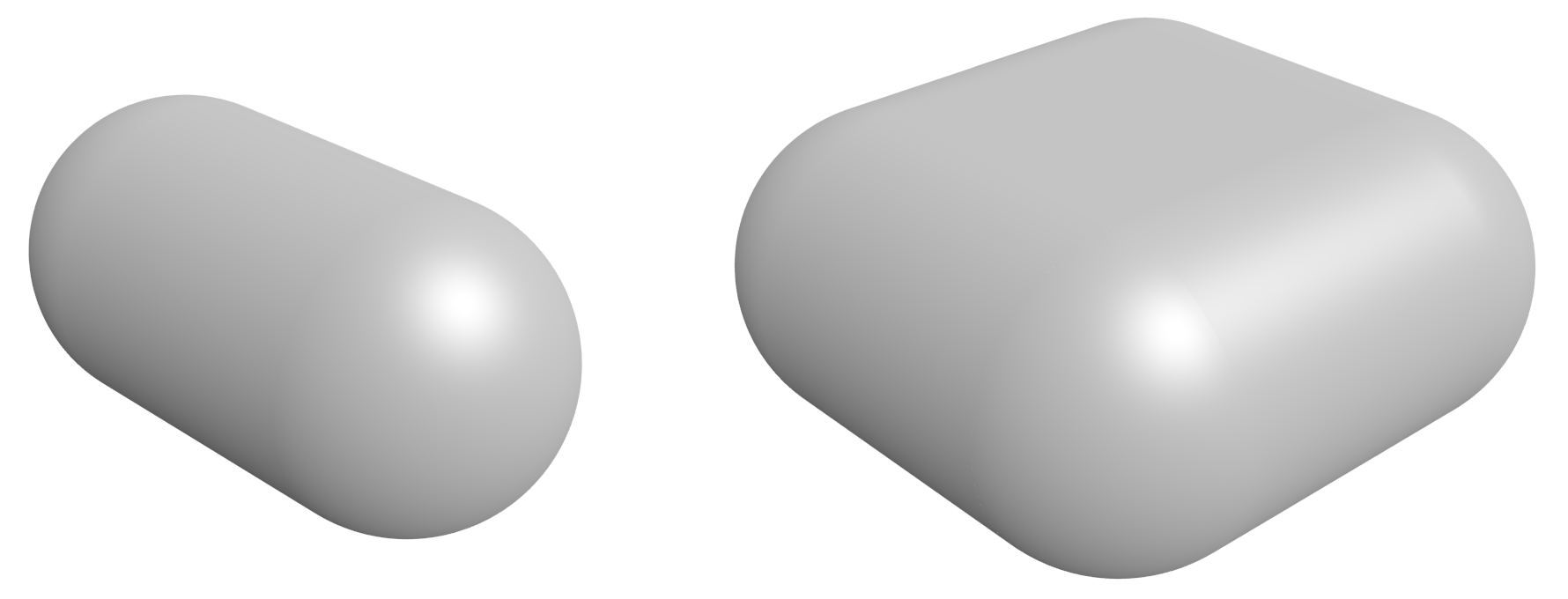}
	\caption{Minkowski sum of a line segment and a unit sphere (on the left hand side), and of a unit square and a sphere (on the right hand side).}
	\label{fig:minkowski}
\end{figure}

\begin{proof}[Proof of Theorem~\ref{thm:everythingispossible}] We use induction on $d_k$ to demonstrate the result. Our induction base is lower dimensional examples discussed earlier. For all increasing sequences of positive numbers $(d_1, \dots, d_k)$ with $d_k \leq 2$  we have found the relevant examples. They are realised by a point, line segment, disk and triangle.

Assume that our assertion is proven for all sequences $(d_1,\dots,d_k)$ with $d_k \leq m$. We will show that the statement is true for $d_k = m+1$. Choose an arbitrary sequence $d = (d_1, \dots, d_{k})$, where $d_k = m+1$. If $d=(d_k)$, the sequence is realised by the Euclidean  unit ball in $\R^{m+1}$. If the sequence contains more than one number, consider the truncated sequence $d' = (d_1,d_2,\dots, d_{k-1})$. Since $d_{k-1}<d_k$, we have $l:= d_{k-1} \leq m$, and there exists a compact convex set $Q\subset \R^l$ that realises the sequence $d'$ in $l=d_{k-1}$-dimensional space. We embed the set $Q$ in the $m+1$-dimensional space by letting $Q' := Q\times \{0_{m+1-l}\}$. Observe that since the definition of the face is algebraic, the facial pattern of the set $Q'$ is identical to the one of $Q$. Let $B$ be the unit ball in $\R^{m+1}$. We let
$$
C:= B + Q'
$$
and claim that $d$ is the facial pattern of $C$.

From Lemma~\ref{lem:sumoffaces} every face of $C$ can be represented as the sum of faces of $Q'$ and $B$. Since the only faces of $B$ are the set itself and the singletons on the boundary, the only possible dimensions of the faces of the set $C$ can come from the sequence $(d_1, \dots, d_k)$.  To show that no facial dimensions are lost, observe that if $e$ denotes the unit vector  $(0,0,\dots, 1)\in B$, then the set $\{e\}+Q'$ is a face of $C$ (hence all its faces are also faces of $C$). Indeed, for  the hyperplane $H  = \{x\, |\, \langle e, x\rangle = x_{m+1} =  1\} $ supports $C$ (notice that for every $x=q+b\in C$ with $q\in Q'$ and $b\in B$ we have $x_{m+1} = 0 + q_{m+1}\leq 1$), moreover,
\begin{align*}
H\cap C & = \{q+b\,|\, q\in Q', b\in B, q_{m+1}+b_{m+1} = 1\}\\
& =\{q+b\,|\,q\in Q', b\in B, b_{m+1} = 1\} \\
& = \{e\} + Q'.
\end{align*}
It is not difficult to observe (e.g., see \cite[Section~18]{Rockafellar}) that any supporting hyperplane slices off a face from a convex set, hence, $F = \{e\}+Q'\lhd C$. This face is linearly isomorphic to $Q$, and hence the facial structure of $F$ conicides with the facial structure of $Q$, giving all possible dimensions of faces from the sequence $d'$. The face of the maximal dimension $m+1$ is given by the set $C$ itself, as it has a nonempty interior (take any point from $Q'$ and sum it with an open ball).
\end{proof}

\section{Fractal convex sets}

Observe that polytopes not only possess faces of all possible dimensions, but their faces are also arranged in a very regular fashion: the union of the edges of a polytope is a one-dimensional set (here we refer to a general notion of Hausdorff dimension, rather than the dimension of the affine hull that is useful for convex sets), the union of all two dimensional faces is two dimensional, and so on. More generally, the union of all faces of a polytope of a given dimension is a set of the same dimension. This is not the case for a more general setting: for instance, the dimension of the union of all extreme points of a Euclidean ball in $\R^n$ is $n-1$, a stark contrast with the polyhedral case. Hence it is natural to study the dimension of the unions of equidimensional faces. The purpose of this section is to present some examples  which emerged from the discussions during the MATRIX program, namely nontrivial sets with fractal facial structure and hence noninteger dimensions of the said unions; these form the foundation for our ongoing research on this topic.


Some work on fractals and convexity has been done before (see the recent work \cite{IFSConvex} and references therein), but we are not aware of any references studying the particular problems that we propose here. We focus on two examples of convex sets that are generated in a natural way by spherical fractals. The finite root system and Coxeter system are fundamental concepts in Lie algebras, which is very important in many branches of mathematics. Given a finite root system, there is a natural associated finite Coxeter group, which is the Weyl group. People in the field of geometric group theory consider finite Coxeter groups are well-studied and explained in liberature, see \cite{humph}. Therefore, we are more interested in the behaviours of infinite Coxeter groups. One such fractal comes from a recent work \cite{Tian} by one of our co-authors (curiously from the study of infinite Coxeter groups), another one is constructed via projecting the Sierpinski triangle onto the unit sphere.

We first consider a fractal set on a sphere and then take its convex hull, hence generating a convex set. Our first example is constructed in a similar way to the Apollonian gasket: we take the unit sphere and construct a tetrahedron whose edges touch the sphere (see Fig.~\ref{fig:gasket-construction}),
\begin{figure}[ht]
	\centering
	\includegraphics[height = 120pt]{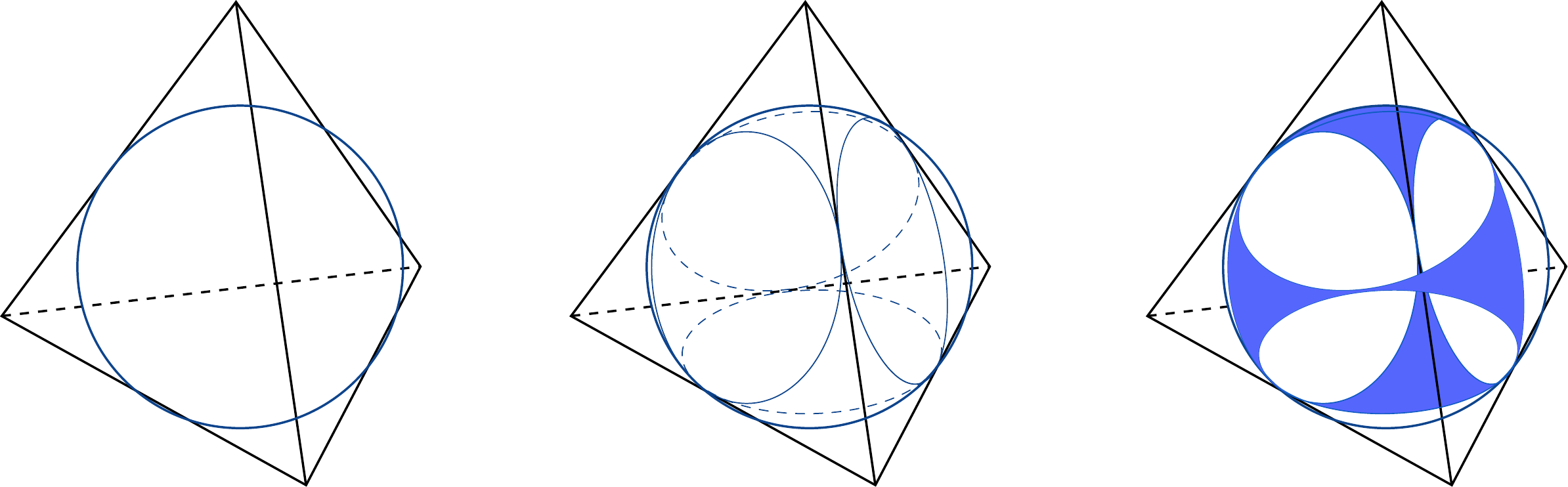}
	\caption{Construction of the spherical gasket.}
	\label{fig:gasket-construction}
\end{figure}
then consider the intersection of the sphere with the tetrahedron. After that, we continue slicing off spherical caps in such a way that they are tangential to the existing slices (see Fig.~\ref{fig:fractals}).
\begin{figure}[ht]
	\centering
	\includegraphics[height = 180pt]{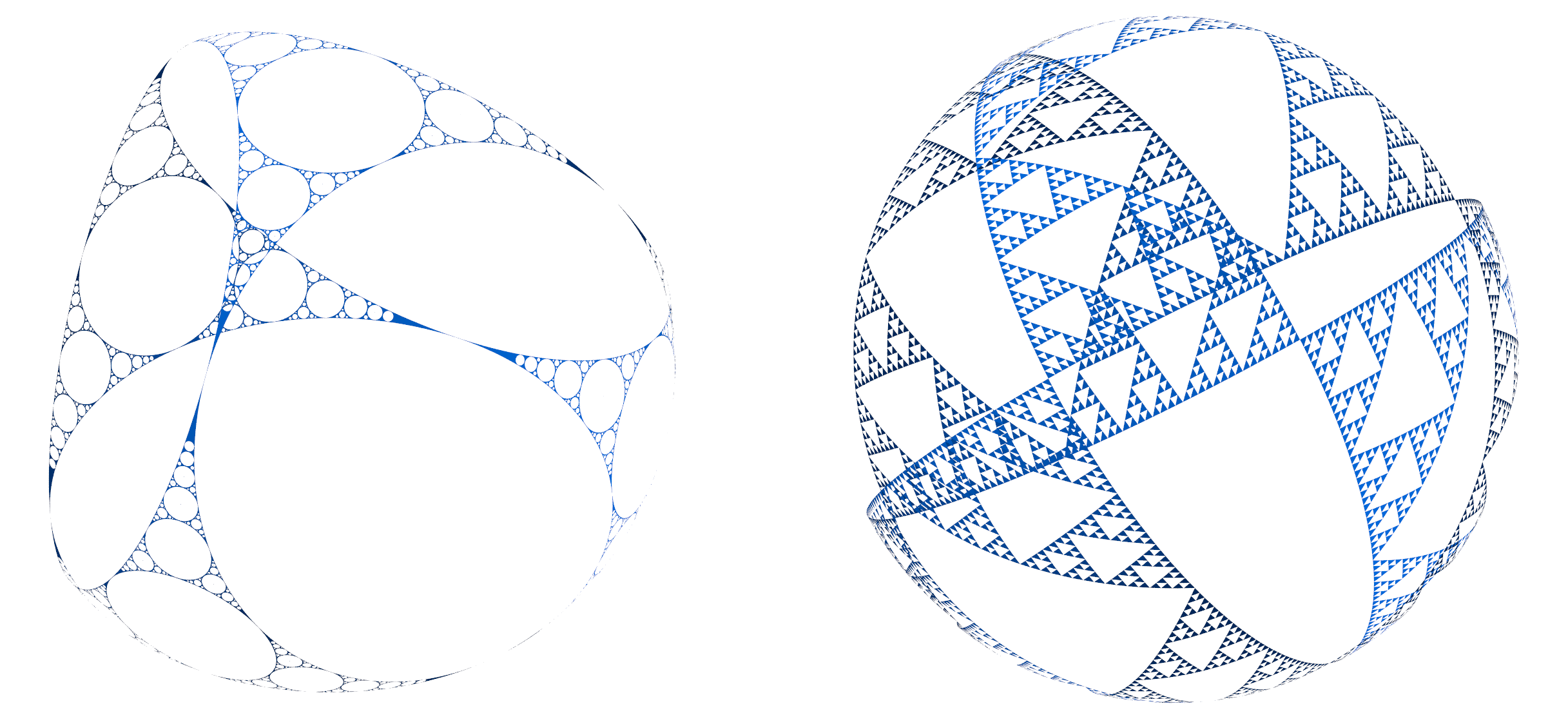}
	\caption{Apollonian gasket on a sphere and Sierpinski triangles.}
	\label{fig:fractals}
\end{figure}
The resulting body is a spherical fractal, which is also a convex set. If we now take its convex hull, the extreme points of this convex set would be exactly the points on the fractal set, with remaining proper faces disks that result from the sliced off spherical caps. Notice that this structure is somewhat similar to the compact convex set obtained as the intersection of the cone of symmetric positive semidefinite matrices of dimension $3\times 3$ with an affine  subspace defined by matrices with a constant trace
$$
C := \Sym_+^3 \cap \{M\, |\, \tr(M) = 1\}.
$$
This set has dimension 5 however.

Algebraically, this particular fractal set is generated by the infinite Coxeter group with following group presentation:
$$G = \langle s_1, s_2, s_3, s_4 \ | \ (s_i)^2 = (s_i s_j)^{\infty} = 1 \rangle$$
The fractal sets are generated by \emph{limit roots}, see \cite{Tian}. Limit roots exhibit peculiar geometric behaviour. Even though Coxeter groups are generated by affine reflections across hyperplanes, when we compute the roots of the group and project them down to a lower dimensional affine hyperplane, the set of limit roots behaves like a fractal set, giving self-similar patterns that cannot be obtained by reflecting across any hyperplanes.

This approach can be applied to constructing other spherical fractals. For instance, one can generalise the Sierpinski carpet by cutting out triangular pieces of the sphere in a similar fashion. The convex set obtained after taking the convex hull of this spherical fractal will have faces of all possible dimensions.

The Hausdorff dimension of the union of the extreme points is non-integer in both cases, and coincides with the dimension of the relevant two-dimensional objects. It would be interesting to study the conditions that can be imposed on the facial dimensions to define good or regular convex sets.

\section{Acknowledgements}
The ideas in this paper were motivated by the discussions that took place during a recent MATRIX program in approximation and optimisation held in July 2016. We are grateful to the MATRIX team for the enjoyable and productive research stay. We would also like to thank the two referees for their insightful corrections and remarks.

\bibliographystyle{plain}
\bibliography{refs}

\end{document}